\documentclass[a4paper,11pt]{amsart}
\usepackage{amsmath}
\usepackage{amsfonts}

\newtheorem{theorem}{Theorem}[section]
\theoremstyle{plain}

\newtheorem{conjecture}{Conjecture}[section]

\newtheorem{definition}{Definition}[section]

\newtheorem{lemma}{Lemma}[section]

\newtheorem{proposition}{Proposition}[section]

\newtheorem{observation}{Observation}[section]
\numberwithin{equation}{section}
\def\func#1{\mathop{\rm #1}\nolimits}
\begin{document}
\title[N. RANGE, DYNAMICS OF A RATIONAL FUNCTION]{NUMERICAL RANGE AND THE
DYNAMICS OF A RATIONAL FUNCTION}
\author{Helena Melo}
\address{Department of Mathematics, Universidade dos A\c{c}ores, Campus de
Ponta Delgada, 9501-801, Ponta Delgada, Portugal}
\email{hmelo@uac.pt}
\author{Jo\~{a}o Cabral}
\address{Department of Mathematics, Universidade dos A\c{c}ores, Campus de
Ponta Delgada, 9501-801, Ponta Delgada, Portugal}
\email{jcabral@uac.pt}

\begin{abstract}
Sometimes we obtain attractive results when associating facts to simple
elements. The goal of this work is to introduce a possible alternative in
the study of the dynamics of rational maps. In this study we use the family
of maps $f(x)=\frac{x^{2}-a}{x^{2}-b}$, making some associations with the
matrix $A=\left(
\begin{array}{cc}
1 & -a \\
1 & -b%
\end{array}%
\right) $ of its coefficients.
\end{abstract}

\maketitle

\section{Introduction\label{seccao 1}}

The main goal of this article is to present an alternative tool to study the
dynamics of a real rational function, using results from the Numerical Range
Theory, and has two main parts. The first, comprising Sections 2-3, is
concerned to the adequation of the Numerical Range Theory to the data
provided from rational maps. As described in Milnor\cite{Milnor 1993}, each
map $f$, in the space \textit{Rat}$_{2}$, can be expressed as a ratio
\begin{equation*}
f(z)=\frac{p(z)}{q(z)}=\frac{a_{0}z^{2}+a_{1}z+a_{2}}{b_{0}z^{2}+b_{1}z+b_{2}%
},
\end{equation*}%
where $a_{0}$ and $b_{0}$ are not both zero and $p(z)$, $q(z)$ have no
common root. Milnor\cite{Milnor 1993} states that we can obtain a roughly
description of the topology of this space \textit{Rat}$_{2}$ that can be
identified with the Zariski open subset of \textit{complex projective 5-space%
} consisting of all points
\begin{equation*}
\left( a_{0}:a_{1}:a_{2}:b_{0}:b_{1}:b_{2}\right) \in CP^{5},
\end{equation*}%
for which the \textit{resultant}
\begin{equation*}
\text{\textit{res}}(p,q)=\det \left(
\begin{array}{cccc}
a_{0} & a_{1} & a_{2} & 0 \\
0 & a_{0} & a_{1} & a_{2} \\
b_{0} & b_{1} & b_{2} & 0 \\
0 & b_{0} & b_{1} & b_{2}%
\end{array}%
\right)
\end{equation*}%
is non-zero. Taking $z=x+i0$ and $a_{0}=1$, $b_{0}=1$, $a_{2}=-a+i0$, $%
b_{2}=-b+i0$, $a_{1}=0$, $b_{1}=0$, with $x,$ $a,$ $b$ real numbers, we
obtain
\begin{equation*}
B=\left(
\begin{array}{cccc}
1 & 0 & -a & 0 \\
0 & 1 & 0 & -a \\
1 & 0 & -b & 0 \\
0 & 1 & 0 & -b%
\end{array}%
\right) \text{,}
\end{equation*}%
\textit{res}$(p,q)=\det $ $B$. This matrix $B$ is associated to the real
rational map $f(x)=\left( x^{2}-a\right) /\left( x^{2}-b\right) $. The map $%
f $ will be the one that we will use in our results associated to the matrix
$B $.

The second part of this article, comprising Sections 4-5, shows how can we
apply the Numerical Range Theory to the dynamics of the map $f$,
establishing the relation between some partitions of an ellipse $\Omega $,
and the symbolic space generated by the partition of the domain of $f$ in
real intervals. Moreover, we launch a conjecture that could be a path to
generate, in the future, an extension of the usual symbolic space applied to
rational maps, allowing us to describe much better the dynamics of this maps.

\section{Numerical Range Theory}

The classical numerical range of a square matrix $M_{n}$, with complex
numbers elements, is the set $W(M_{n})=\left\{ u^{\ast }M_{n}u,\text{ }u\in
S(\mathbb{C}^{n})\right\} $ , with $S(\mathbb{C}^{n})$ the unit sphere, $u$
is a vector in $\mathbb{C}^{n}$ and $u^{\ast }$ is the transpose conjugate
of $u$. The numerical range $W(M_{n})$ also can be defined as the the image
of the \textit{Rayleight quotient }$R_{M_{n}}(u)=u^{\ast }M_{n}u/u^{\ast }u$%
, $u\neq 0$. The set $W(M_{n})$ is closed and limited, and it is also a
subset of the Gaussian $\mathbb{C}$ plane. Toeplitz\cite{Toeplitz 1918} and
Hausdorff\cite{Hausdorff 1919} proved that $W(M_{n})$ is a convex region.
From Kippenhahn\cite{Kippenhahn 1951}, the boundary of the numerical range, $%
\partial W(M_{n})$ is a piecewise algebraic curve. In the particular case of
a square matrix $M_{2}$, with eigenvalues $\lambda _{1}$, $\lambda _{2}$, $%
W(M_{2})$ is a subset limited by an ellipse with foci in $\lambda _{1}$ and $%
\lambda _{2}$, result known as Elliptical Range Theorem, see Li\cite{Li 1996}%
.

\begin{proposition}
If $H_{M_{n}}=\left( M_{n}+M_{n}^{\ast }\right) /2$ and $S_{M_{n}}=\left(
M_{n}-M_{n}^{\ast }\right) /2$ are the Hermitian and skew-Hermitian parts of
$M_{n}$, respectively, then $\func{Re}(W(M_{n}))=W(H_{M_{n}})$ and $\func{Im}%
(W(M_{n}))=W(S_{M_{n}})$.
\end{proposition}

\begin{proof}
The proof can be found in Melo\cite{Melo 1999}.
\end{proof}

\begin{theorem}
\label{Teorema convex hull}To every complex matrix $%
M_{n}=H_{M_{n}}+S_{M_{n}} $ through the equation $k_{M_{n}}(\alpha
_{1},\alpha _{2},\alpha _{3})\equiv \det (\alpha _{1}H_{M_{n}}-i\alpha
_{2}S_{M_{n}}+\alpha _{3}I_{n})=0$ is associated a curve of class $n$ in
homogeneous line coordinates in the complex plane. The convex hull of this
curve is the numerical range of the matrix $M_{n}$.
\end{theorem}

\begin{proof}
Adapting the proof in Kippenhahn\cite{Kippenhahn 1951} we have the desired
result.
\end{proof}

\section{Merging $f(x)$ in $W(M_{n})$}

Hwa-Long Gau\cite{Gau 2006} states that we can obtain from a $4\times 4$
matrix an elliptical numerical range, thus we could use $B$ as defined in
section \ref{seccao 1}, but we can simplify our results if we use a smaller
matrix, $A_{2}$, trough a result in linear algebra. The new matrix $%
A_{2}=\left(
\begin{array}{cc}
1 & -a \\
1 & -b%
\end{array}%
\right) $ will produce equivalent results as the obtained from $B$.

\begin{lemma}
\label{Decomposicao unitaria}The matrix $B\ $is unitary decomposable in
\begin{equation*}
\left(
\begin{array}{cc}
A_{2} & \mathbb{O}_{2} \\
\mathbb{O}_{2} & A_{2}%
\end{array}%
\right)
\end{equation*}%
a block diagonal matrix.
\end{lemma}

\begin{proof}
In order to prove this result it is sufficient to find an unitary matrix $E$
such that%
\begin{equation*}
E^{\ast }BE=\left(
\begin{array}{cc}
A_{2} & \mathbb{O}_{2} \\
\mathbb{O}_{2} & A_{2}%
\end{array}%
\right) \text{.}
\end{equation*}%
With some computation we can see that
\begin{equation*}
E=\left(
\begin{array}{cccc}
1 & 0 & 0 & 0 \\
0 & 0 & 1 & 0 \\
0 & 1 & 0 & 0 \\
0 & 0 & 0 & 1%
\end{array}%
\right) .
\end{equation*}
\end{proof}

\begin{proposition}
$W(A_{2})=W(B)$
\end{proposition}

\begin{proof}
By lemma \ref{Decomposicao unitaria} $E^{\ast }BE=A_{2}\oplus A_{2}$ and
using the properties of numerical range we have
\begin{equation*}
W(E^{\ast }BE)=\text{\textit{convex hull}}\{W(A_{2})\cup W(A_{2})\}=\text{%
\textit{convex hull}}\{W(A_{2})\}.
\end{equation*}%
The \textit{convex hull} of a convex set is itself, then $W(E^{\ast
}BE)=W(A_{2})$. But the numerical range of $B$ is invariant under unitary
transformations, see Kippenhahn\cite{Kippenhahn 1951}, so $W(B)=W(A_{2})$.
\end{proof}

In our study we use $f(x)=(x^{2}-a)/(x^{2}-b)$, with $a>0$, $b>0$ and $a>b$.
Such map takes all real axis with exceptions $\pm \sqrt{b}$ on $(-\infty
,1)\cup \lbrack \frac{a}{b},+\infty )$. With $\Lambda =\mathbb{R}\backslash
\{\pm \sqrt{b}\}\times $ $(-\infty ,1)\cup \lbrack \frac{a}{b},+\infty )$ we
can define the graphic of $f$, $graph(f)$, as the pair $(x$, $f(x))\in
\Lambda $ and $\theta :\mathbb{R}\backslash \{\pm \sqrt{b}\}\longrightarrow
\Lambda $.

\begin{definition}
\label{Definicao funcoes}Let $C=\left\{ v\in \mathbb{C}^{2}:v=(x\text{, }%
if(x)\right\} $ and
\begin{equation*}
\Psi =\left\{ z\in \mathbb{C}:z=\frac{v^{\ast }A_{2}v}{v^{\ast }v},v\neq
0\right\} .
\end{equation*}%
We define $V:\Lambda \longrightarrow C$ as $(x$, $f(x))\longmapsto (x$, $%
if(x))$ and $\Xi :C\longrightarrow $ $\Psi $.
\end{definition}

By definition \ref{Definicao funcoes} the image of $(x$, $f(x))$ is $%
z=\left( v^{\ast }A_{2}v\right) /v^{\ast }v$ by $\Xi \circ V$.

\begin{proposition}
\label{criacao de psi em wa2}$\Psi $ $\subset W(A_{2})$
\end{proposition}

\begin{proof}
By the definition of $\Psi $ and $W(A_{2})$, using \textit{Rayleight quotient%
}, the result follows.
\end{proof}

From proposition \ref{criacao de psi em wa2} we know that $\Psi $ is a
subset of $W(A_{2})\subset \mathbb{C}$, and using definition \ref{Definicao
funcoes}, we can calculate the elements $z\in \Psi $, and as they were
defined, they will become a function of $x$. After some calculations we have
\begin{equation*}
z(x)=\frac{-a^{2}b+\left( 2a+b\right) bx^{2}-3bx^{4}+x^{6}+i\left(
a+1\right) \left( -abx+(b+a)x^{3}-x^{5}\right) }{a^{2}+\left(
b^{2}-2a\right) x^{2}+\left( 1-2b\right) x^{4}+x^{6}}.
\end{equation*}%
So, the $z\in \Psi $ is a function such that $z(x)=g(x)+ih(x)$, with $x\in
\mathbb{R}$. Some elementary calculus show us that $g(x)$ and $h(x)$ are
real rational continuous functions in $\mathbb{R}$,$\ $therefore $z(x)$ is
continuous in $\mathbb{C}$. We call some attention to the fact that $z(\pm
\sqrt{b})$ and $z(\infty )$ exists and are well defined in $\mathbb{C}$.

\begin{observation}
We have $z(x+1)=z(x)$ for
\begin{equation*}
x=\frac{1}{2}\left( -1\pm \sqrt{1+6a-2b\pm 2\sqrt{4a+9a^{2}-10ab+b^{2}}}%
\right) \text{.}
\end{equation*}
\end{observation}

Let $x_{2}$ and $x_{12}$ be the values where $f(x)=0$. If we calculate $%
z(0,a/b)$; $z(x_{2},0)$; $z(x_{12},0)$; $z(x,x)$; $z(x,-x)$ we obtain four
different points of $\Psi $. With some elementary algebra we calculated the
ellipse that contain this four points. This ellipse is
\begin{equation*}
\Omega =\left\{ \frac{\left( x-\frac{1-b}{2}\right) ^{2}}{\left( \frac{1+b}{2%
}\right) ^{2}}+\frac{y^{2}}{\left( \frac{1+a}{2}\right) ^{2}}=1\text{, }(x%
\text{, }y)\in \mathbb{R}^{2}\text{, }a>b,a>0,b>0\right\} \text{,}
\end{equation*}%
with $1+b$ and $1+a$ the minor and major axis length of $\Omega $,
respectively.

Moreover, when we use all points $(x,$ $f(x))$ they will fall in $\Omega $
under transformation by $z$.

\begin{lemma}
\label{lema de z a formar omega}Let $z(x)=g(x)+ih(x)$, then the pair $(\func{%
Re}(z),\func{Im}(z))$ satisfies $\Omega .$
\end{lemma}

\begin{proof}
We obtain this result replacing in the equation of $\Omega $, $x$ by $\func{%
Re}(z)$ and $y$ by $\func{Im}(z).$
\end{proof}

\begin{proposition}
\label{proposicao do si em psi}If $S_{i}\in \Psi $ there are, at least one $%
x_{i}$ such that $z(x_{i})=S_{i}$.
\end{proposition}

\begin{proof}
Since $z(x)=g(x)+ih(x)$ is a continuous function in $\mathbb{C}$ and by the
lemma \ref{lema de z a formar omega} the result follows.
\end{proof}

Then we conclude that $\Psi $ can be represented by the ellipse with
equation $\Omega $.

Since $\Omega $ is constructed in the space $\mathbb{R}^{2}$ and this space
is isomorphic to $\mathbb{C}$, when we refer to an element $z$ $\in \Omega $
it can understood has a vector in $\mathbb{R}^{2}$ or a complex number in
the plane $\mathbb{C}$.

There are relations between the functions $f$ and $g$ that we can observe,
described in the following lemmas. The proofs are omitted because they
result from straight calculus.

\begin{lemma}
\label{lema do zero de f e max relat g}If $x_{0}$ is a zero of $f(x)$, then $%
g(x_{0})$ is a relative maximum of $g(x)$.
\end{lemma}

\begin{lemma}
\label{lema das decont f e min de g}If $x_{0}$ is a relative minimum of $%
f(x) $ or $x_{0}$ is a discontinuity value of $f(x)$, then $g(x_{0})$ is a
relative minimum of $g(x)$.
\end{lemma}

There are similar relations between $h(x)$ and $f(x)$.

Follows some results relating $f(x)$ to $\Omega $.

\begin{lemma}
\label{lema do ponto fixo e do vertice}Let $f(x_{0})=\pm x_{0}$, then $\Xi
\circ V(x_{0},f(x_{0}))$ is vertex of $\Omega $.
\end{lemma}

\begin{proof}
If $f(x_{0})=x_{0}$, by $V$ we have $\left( x_{0}\text{, }ix_{0}\right) $.
So
\begin{equation*}
\Xi (\left( x_{0}\text{, }ix_{0}\right) )=\frac{\left(
\begin{array}{cc}
x_{0} & -ix_{0}%
\end{array}%
\right) \left(
\begin{array}{cc}
1 & -a \\
1 & -b%
\end{array}%
\right) \left(
\begin{array}{c}
x_{0} \\
ix_{0}%
\end{array}%
\right) }{\left(
\begin{array}{cc}
x_{0} & -ix_{0}%
\end{array}%
\right) \left(
\begin{array}{c}
x_{0} \\
ix_{0}%
\end{array}%
\right) }=\frac{1-b}{2}-i\frac{1+a}{2}
\end{equation*}

And if we look at the equation of $\Omega $, we see that $\left( \frac{1-b}{2%
},-\frac{1+a}{2}\right) $ is a vertex of $\Omega $.

If $f(x_{0})=-x_{0}$ we obtain another vertex of $\Omega $ in a similar way,
which is $\left( \frac{1-b}{2},\frac{1+a}{2}\right) $.
\end{proof}

\begin{lemma}
\label{lema das descontinuidades e do vertice}The discontinuities of $f(x)$
and the values where $f(x)$ has a minimum are transformed by $\Xi \circ
V\circ \theta $ in the vertex $(-b,0)$ of $\Omega $, and the roots of $f(x)$
and the $\infty $ are transformed by $\Xi \circ V\circ \theta $ in the
vertex $(1,0)$ of $\Omega $.
\end{lemma}

\begin{proof}
Since $z(x)=g(x)+ih(x)=v^{\ast }A_{2}v/v^{\ast }v$, $v=(x,if(x))$ is a
continuous function in $\mathbb{C}$, we have $z(\pm \sqrt{b})=-b$ and $%
z(\infty )=1$.
\end{proof}

\section{Partitions of $\Omega $}

Let $x_{1}$, $x_{2}$, $x_{5}$, $x_{6}$, $x_{7}$, $x_{8}$, $x_{9}$, $x_{12}$,
$x_{13}$ be the solutions of $g^{\prime }(x)=0$. By lemma \ref{lema das
decont f e min de g}, and lemma \ref{lema do zero de f e max relat g}, and
considering the order of real axis, we will have $x_{2}$ and $x_{12}$ as
zeros of $f(x)$; $x_{5}$ and $x_{9}$ as the discontinuities of $\ f(x)$ and $%
x_{7}=0$. All this values have image from $z(x)$, lemma \ref{lema de z a
formar omega}, including the infinity, being related by $%
z(x_{2})=z(x_{12})=z(\infty )$ and equal to vertex $(1,0)$ in $\Omega $, see
lemma \ref{lema das descontinuidades e do vertice}, $z(x_{5})=z(x_{9})=z(0)$
and equal to vertex $(-b,0)$ in $\Omega $. Related to the real axis, $%
z(x_{1})$ is symmetric to $z(x_{13})$ and $z(x_{6})$ is symmetric to $%
z(x_{8})$ in $\Omega $. Where are the missing $x_{3}$, $x_{4}$, $x_{10}$, $%
x_{11}$ ? They will be the values such that $z(x_{1})=z(x_{11})$, $%
z(x_{6})=z(x_{10})$, $z(x_{8})=z(x_{4})$ and $z(x_{13})=z(x_{3})$.

Using this special values $x_{i}$, $i=1,..,13$, with order $x_{i}<x_{i+1}$,
we can define a partition function $pa$, as
\begin{equation*}
pa(x)\underset{x\in \mathbb{R}}{=}\left\{
\begin{array}{ll}
I_{1}\text{,} & if\text{ }x<x_{1} \\
I_{i}\text{,} & if\text{ }x_{i-1}<x<x_{i}\text{ with }2\leq i\leq 13 \\
I_{14}\text{,} & if\text{ }x>x_{13}%
\end{array}%
\right. .
\end{equation*}

Now we will create partitions in $\Omega $ using the images $z(x_{i})$, $%
i=1,..,13$ in $\Omega $. Here, we ask attention for one particular aspect of
$z$, see proposition \ref{proposicao do si em psi}. Some intervals $I_{i}$
will be transformed in the same arc of $\Omega $. The only thing that will
distinguish them is the orientation and the origin of its end points.

\begin{definition}
Let $S_{i}=z(x_{i})$, we define $arc(S_{i},S_{i+1})$ as the arc of $\Omega $
starting at $S_{i}$ and ending at $S_{i+1}$, with counterclockwise
orientation.
\end{definition}

We define $pa_{\Omega }$, a partition function, as:
\begin{equation*}
pa_{_{\Omega }}(w)\underset{w\in \mathbb{C}}{=}\left\{
\begin{array}{ll}
J_{1}\text{,} & if\text{ }w\in \text{ }arc(z(\infty )\text{, }z(x_{1}))\text{
} \\
J_{i}\text{,} & if\text{ }w\in \text{ }arc(z(x_{i-1})\text{, }z(x_{i}))\text{
with }2\leq i\leq 13 \\
J_{14}\text{,} & if\text{ }w\in \text{ }arc(z(x_{14})\text{, }z(\infty ))%
\end{array}%
\right. .
\end{equation*}

The functions $pa(x)$ and $pa_{\Omega }(w)$ are related by
\begin{equation*}
z(I_{i})=\left\{
\begin{array}{ll}
J_{1} & if\text{ }i=1 \\
-J_{i} & if\text{ }2\leq i\leq 6 \\
J_{i} & if\text{ }i=7\text{ or }i=8 \\
-J_{i} & if\text{ }9\leq i\leq 13 \\
J_{14} & if\text{ }i=14%
\end{array}%
\right. \text{,}
\end{equation*}%
thus we can build a matrix $T_{14}$ of the transformation $pa_{\Omega
}(z(pa))$,%
\begin{equation*}
T_{14}=\left[
\begin{array}{cc}
N_{7} & \mathbb{O}_{7} \\
\mathbb{O}_{7} & N_{7}%
\end{array}%
\right] \text{,}
\end{equation*}%
with
\begin{equation*}
N_{7}=\left[
\begin{array}{ccccccc}
1 & 0 & 0 & 0 & 0 & 0 & 0 \\
0 & -1 & 0 & 0 & 0 & 0 & 0 \\
0 & 0 & -1 & 0 & 0 & 0 & 0 \\
0 & 0 & 0 & -1 & 0 & 0 & 0 \\
0 & 0 & 0 & 0 & -1 & 0 & 0 \\
0 & 0 & 0 & 0 & 0 & -1 & 0 \\
0 & 0 & 0 & 0 & 0 & 0 & 1%
\end{array}%
\right] .
\end{equation*}%
It is easy to see that $T_{14}.T_{14}=I_{14}$, $det(T_{14})=1$, and it is an
involutary matrix.

\section{Dynamics of $f(x)$}

Now we have a new tool to study the dynamics of $f$ using a symbolic space.
Using $\Omega $ to study the behavior of $f$ we will have the same
advantages that we would have when studying the behavior of second degree
polynomials functions in the unit circle.

If we define a symbolic space using the partitions created by the function $%
pa$ in the real axis we will have the problem of dealing with the
discontinuities of the function $f$ and the infinity itself. So, profiting
that $z$ is a continuous complex function in $\mathbb{C\cup \{\infty \}}$
this problem will vanish.

We can build two distinct symbolic spaces. The first will be the classical
association between the intervals produced by $pa$ in the real axis, see
Milnor\cite{Milnor 1993} for further reference, using the domain of the
function $f$, and considering an alphabet $\mathcal{A}$ with designations $%
I_{i}$ for each interval, we will have a symbolic space $\Sigma _{c}=$ $%
\mathcal{A}^{\mathbb{N}}$. The second will be constructed as we consider the
alphabet $\mathcal{B}=\{J_{1},...,J_{14}\}$, and the set $\Sigma =\mathcal{B}%
^{\mathbb{N}}$ of symbolic sequences on the elements of $\mathcal{B}$,
introducing the map $spa:\mathbb{R\cup \{\infty \}}\longrightarrow \mathcal{B%
}$.

\begin{conjecture}
The symbolic dynamics of $f$ does not change if we use $\Sigma $ instead of $%
\Sigma _{c}$.
\end{conjecture}

Both spaces are connected by the transformation matrix $T_{14}$ and doing
some calculus in matrix algebra, since this matrix is an involutary matrix,
we could get the result. All computations in our work points in that
direction. But we are still working in a suitable proof of this result.
Moreover, $\Sigma $ will work as an extension of $\Sigma _{c}$.

It means that we can identify the periodic orbits in the same values of $a$
and $b$ as we use both spaces\ $\Sigma $ and $\Sigma _{c}$. For example for
the values $a=4.01$ and $b=2.5$ the critical orbit of $f$ is periodic in
both spaces, such as all the others values of periodicity found in our
research. But they are many new sequences in $\Sigma $ that needs more work
to full understood slightly changes caused by the obliteration of $\infty $.
We are in the way.

This work, when finished, will imply a sequential work in the kneading
theory and further study of the entropy of rational real maps.

\end{document}